\documentclass[12pt,reqno]{amsart}

\usepackage{amsmath,amssymb, latexsym }
\usepackage{graphics,color}

\usepackage[mathscr]{eucal}

\newdimen\unit\newdimen\psep\newcount\nd\newcount\ndx\newbox\dotb\newbox\ptbox
\newdimen\dx\newdimen\dy\newdimen\dxx\newdimen\dyy\newdimen\hgt 
\newdimen\xoff\newdimen\yoff
\newcommand\clap[1]{\hbox to 0pt{\hss{#1}\hss}}
\newcommand\vdisk[1]{{\font\dotf=cmr10 scaled #1\dotf.}}
\newcommand\varline[2]{\setbox\dotb\hbox{\vdisk{#1}}\xoff=-.5\wd\dotb
\wd\dotb=0pt\yoff=-.5\ht\dotb\psep=#2\ht\dotb}
\newcommand\varpt[1]{\setbox\ptbox\clap{\vdisk{#1}}\setbox\ptbox
\hbox{\raise-.5\ht\ptbox\box\ptbox}}
\newcommand\cpt{\copy\ptbox}
\newcommand\point[3]{\rlap{\kern#1\unit\raise#2\unit\hbox{#3}}}
\newcommand\setnd[4]{\dx=#3\unit\advance\dx-#1\unit\divide\dx by\psep
\dy=#4\unit\advance\dy-#2\unit\divide\dy by\psep \multiply\dx
by\dx\multiply\dy by\dy\advance\dx\dy\nd=1\advance\dx-1sp
\loop\ifnum\dx>0\advance\dx-\nd sp\advance\nd1\advance\dx-\nd
sp\repeat}

\newcommand\dline[5]{{\nd=#5\hgt=#2\unit\dx=#3\unit\advance\dx-#1\unit
\divide\dx by\nd\dy=#4\unit\advance\dy-#2\unit\divide\dy by\nd
\advance\hgt\yoff\rlap{\kern#1\unit\kern\xoff\loop\ifnum\nd>1\advance\nd-1
\advance\hgt\dy\kern\dx\raise\hgt\copy\dotb\repeat}}}

\newcommand\qellip[4]{{\setnd{0}{0}{#3}{#4}\dx=\unit\dy=0pt\raise\yoff\rlap{%
\kern#1\unit\kern\xoff\raise#2\unit\hbox{\loop\ifnum\dx>0\rlap{\kern#3\dx
\raise#4\dy\copy\dotb}\hgt=\dx\divide\hgt
by\nd\advance\dy\hgt\hgt=\dy \divide\hgt
by\nd\advance\dx-\hgt\repeat\rlap{\raise#4\dy\copy\dotb}}}}}
\newcommand\bez[6]{{\setnd{#1}{#2}{#3}{#4}\ndx=\nd\setnd{#3}{#4}{#5}{#6}
\ifnum\ndx>\nd\nd=\ndx\fi\dx=#3\unit\advance\dx-#1\unit\dy=#4\unit
\advance\dy-#2\unit\dxx=#5\unit\advance\dxx-#1\unit\dyy=#6\unit\advance
\dyy-#2\unit\advance\dxx-2\dx\advance\dyy-2\dy\divide\dxx
by\nd\divide\dyy
by\nd\advance\dx.25\dxx\advance\dy.25\dyy\divide\dx
by\nd\divide\dy by\nd \multiply\nd
by2\dx=100\dx\dy=100\dy\dxx=100\dxx\dyy=100\dyy\divide\dxx by\nd
\divide\dyy
by\nd\hgt=#2\unit\raise\yoff\rlap{\kern#1\unit\kern\xoff
\raise\hgt\copy\dotb\loop\ifnum\nd>0\advance\nd-1\advance\hgt0.01\dy
\kern0.01\dx\raise\hgt\copy\dotb\advance\dx\dxx\advance\dy\dyy\repeat}}}
\newcommand\ptu[3]{\point{#1}{#2}{\cpt\raise1ex\clap{$\scriptstyle{#3}$}}}
\newcommand\ptd[3]{\point{#1}{#2}{\cpt\raise-1.8ex\clap{$\scriptstyle{#3}$}}}
\newcommand\ptr[3]{\point{#1}{#2}{\cpt\raise-.4ex\rlap{$\ \scriptstyle{#3}$}}}
\newcommand\ptl[3]{\point{#1}{#2}{\cpt\raise-.4ex\llap{$\scriptstyle{#3}\ $}}}
\newcommand\ptlu[3]{\point{#1}{#2}{\raise.8ex\clap{$\scriptstyle{#3}$}}}
\newcommand\ptld[3]{\point{#1}{#2}{\raise-1.6ex\clap{$\scriptstyle{#3}$}}}
\newcommand\ptlr[3]{\point{#1}{#2}{\raise-.4ex\rlap{$\,\scriptstyle{#3}$}}}
\newcommand\ptll[3]{\point{#1}{#2}{\raise-.4ex\llap{$\scriptstyle{#3}\,$}}}

\newcommand\thnline{\varline{400}{.6}}

\varpt{2500}\thnline\unit=2em

\newtheorem{thm}{Theorem}

\newtheorem*{vBKlemma}{The van den Berg--Kesten Lemma}

\newtheorem{conj}{Conjecture}

\newtheorem{prob}{Problem}
\newtheorem{lemma}[thm]{Lemma}
\newtheorem{prop}[thm]{Proposition}

\newtheorem{obs}[thm]{Observation}
\theoremstyle{definition}
\theoremstyle{definition}\newtheorem*{defn}{Definition}
\theoremstyle{definition}
\theoremstyle{definition}
\newcommand{\ds}{\displaystyle}

\def\A{\mathcal{A}}
\def\B{\mathcal{B}}
\def\C{\mathcal{C}}

\def\HH{\mathcal{H}}

\def\N{\mathbb{N}}
\def\Pr{\mathbb{P}}

\def\RR{\mathbb{R}}
\def\ZZ{\mathbb{Z}}

\def\a{\mathbf{a}}
\def\b{\mathbf{b}}

\def\le{\leqslant}
\def\ge{\geqslant}

\def\eps{\varepsilon}
\def\->{\rightarrow}
\def\<{\langle}
\def\>{\rangle}

\def\lg{\textup{lg}}
\def\sht{\textup{sh}}

\def\poly{\textup{poly}}

\begin{document}
\title[Bootstrap percolation in two dimensions]{A sharper threshold for bootstrap percolation in two dimensions}

\author{Janko Gravner}
\address{Mathematics Department, University of California, Davis, CA 95616, USA} \email{gravner@math.ucdavis.edu}

\author{Alexander E. Holroyd}
\address{Microsoft Research, 1 Microsoft Way, Redmond, WA 98052, USA; and University of British Columbia, 121-1984 Mathematics Road, Vancouver, BC V6T 1Z2, Canada} \email{holroyd at math.ubc.ca}

\author{Robert Morris}
\address{IMPA, Estrada Dona Castorina 110, Jardim Bot\^anico, Rio de Janeiro, RJ, Brasil} \email{rob@impa.br}
\thanks{Supported by NSF grant DMS 0204376 and the Republic of Slovenia Ministry of Science program
P1-285 (JG); NSERC and Microsoft Research (AEH); a JSPS Fellowship and a
Research Fellowship from Murray Edwards College, Cambridge (RM)}

\begin{abstract}
Two-dimensional bootstrap percolation is a cellular automaton in which sites become `infected' by contact with two or more already infected nearest
neighbours. We consider these dynamics, which can be interpreted as a monotone version of the Ising model, on an $n\times n$ square, with sites
initially infected independently with probability $p$. The critical probability $p_c$ is the smallest $p$ for which the probability that the entire square is eventually infected exceeds $1/2$. Holroyd determined the sharp first-order approximation: $p_c\sim \pi^2/(18\log n)$ as $n\to\infty$. Here we sharpen this result, proving that the second term in the expansion is $-(\log n)^{-3/2+ o(1)}$, and moreover determining it up to a $\poly(\log\log n)$-factor.  The exponent $-3/2$ corrects numerical predictions from the physics literature.
\end{abstract}

\maketitle

\section{Introduction}\label{intro}

Bootstrap percolation is a cellular automaton in which, given a (typically random) initial set of `infected' vertices in a graph $G$, new vertices are infected at each time step if they have at least $r$ infected neighbours. In this paper we shall study two-neighbour bootstrap percolation on the square grid $[n]^2$. We shall determine the second term of the critical threshold for percolation up to a $\poly(\log\log n)$-factor, and hence confirm a conjecture of Gravner and Holroyd~\cite{GH2}.

We begin by defining the bootstrap process, which was introduced by Chalupa, Leath and Reich~\cite{CLR} in 1979. Let $G$ be a graph with vertex set $V(G)$, and for each vertex $v \in V(G)$, let $N(v)$ denote the set of neighbours of $v$ in $G$. Now, given an integer $r \in \N$, and a set of initially infected vertices $A \subset V(G)$, define $A_t$ recursively as follows: set $A_0 = A$, and
$$A_{t+1} \; = \; A_t \cup \big\{v \in V(G) : |N(v) \cap A_t| \ge r \big\}$$
for each integer $t \ge 0$. We say that the vertices of $A_t$ have been infected by time $t$. Let $[ A ] = \bigcup_t A_t$ denote the closure of $A$ under the $r$-neighbour bootstrap process, and say that the set $A$ \emph{percolates} if the entire vertex set is eventually infected, i.e., if $[A] = V(G)$.

We shall be interested in the case in which $A$ is a random subset of $V(G)$. More precisely, let us choose the elements of $A$ independently at random, each with probability $p$, and denote by $\Pr_p$ the corresponding probability measure. Throughout the paper, $A$ will be assumed to be a random subset selected according to this distribution, unless otherwise stated. It is clear that the probability of percolation is increasing in $p$, and so we define the critical probability, $p_c(G,r)$ as follows:
$$p_c(G,r) \; := \; \inf \Big\{ p \,:\, \Pr_p\big(A \textup{ percolates in the $r$-neighbour process on $G$}\big) \ge 1/2 \Big\}.$$
Our aim is to give sharp bounds on $p_c(G,r)$.

Bootstrap percolation has been studied extensively by mathematicians~\cite{AL,d=r=3,CC,Hol,Sch}, as well as by physicists~\cite{ALev, BDGM,GLBD} and sociologists~\cite{Gran,Watts}, amongst others. The bootstrap process was originally introduced in the context of disordered magnetic systems, and may be thought of as a monotone version of the Glauber dynamics of the Ising model. Indeed, if spins are allowed to flip in only one direction (from $-$ to $+$, say), and if they flip only if they have at least $r$ neighbours in state $+$, then one immediately obtains the cellular automaton described above. We refer the interested reader to the applications of bootstrap percolation in~\cite{FSS,M}, and the closely related models studied in~\cite{CM2,DS1,DS2,GG,KS,Rich}.

We focus on the graph $G=[n]^d$ with vertex set $\{1,\ldots,n\}^d$, and with an edge between vertices $u$ and $v$ if and only if $\|u-v\|_1=1$. Aizenman and Lebowitz~\cite{AL} determined the asymptotic behaviour of $p_c([n]^d,2)$ up to multiplicative constants, and Cerf and Cirillo~\cite{CC} (in the crucial case $d = r = 3$) and Cerf and Manzo~\cite{CM} proved the corresponding result for all $d \ge r \ge 2$. The first sharp threshold for bootstrap percolation was proved by Holroyd~\cite{Hol}, who showed that
\begin{equation}\label{h}
p_c([n]^2,2) \; = \; \frac{\pi^2}{18\log n} \,+\, o\left( \frac{1}{\log n} \right).
\end{equation}
This was the first result of its type, and has prompted a flurry of generalizations. Sharp thresholds have since been determined for $p_c([n]^d,r)$ for all fixed $d$ and $r$~\cite{d=r=3,alldr}, for more general update rules in two dimensions~\cite{DCH}, and in high dimensions (i.e., $d = d(n) \to \infty$ sufficiently fast) for the majority rule, i.e., $r = d$~\cite{Maj}, and in the case $r = 2$~\cite{cube}. Some of the techniques from these papers have been used to prove results about the Glauber dynamics of the Ising model~\cite{FSS,M}. The bootstrap process has also been studied on infinite trees~\cite{BPP,BS,FS}, on `locally tree-like' graphs~\cite{Maj}, on the random regular graph~\cite{BP,Svante}, and on the Erd\H{o}s-R\'enyi random graph $G_{n,p}$~\cite{Tomas}.

In this paper we shall study the two-neighbour bootstrap process on the graph $G = [n]^2$ in more detail.  One of the most striking facts about the result \eqref{h} stated above is that it contradicted estimates of $\ds\lim_{n \to \infty} {p_c \log n}$ given by simulations - in fact, such estimates were out by a factor of more than two.  (See, for example,~\cite{GH1} or~\cite{GLBD} for a discussion of the reasons behind these discrepancies.) Gravner and Holroyd~\cite{GH1} gave a rigorous (partial) explanation for this phenomenon, by giving the following improvement of \eqref{h}:
$$p_c([n]^2,2) \; \le \; \frac{\pi^2}{18\log n} \,-\, \frac{c}{(\log n)^{3/2}},$$
where $c > 0$ is a small constant. In \cite{GH2}, the same authors proved an almost matching lower bound for a simpler model (called `local' bootstrap percolation), and conjectured that the upper bound is essentially sharp for the usual bootstrap process.

\begin{conj}[Gravner and Holroyd~\cite{GH2}]\label{GHconj}
For every $\eps > 0$, if $n$ is sufficiently large then
$$p_c([n]^2,2) \; \ge \; \frac{\pi^2}{18\log n} \,-\, \frac{1}{(\log n)^{3/2-\eps}}.$$
\end{conj}

In this paper we shall prove Conjecture~\ref{GHconj} in a slightly stronger form. To be precise, we shall prove the lower bound in the following theorem; the upper bound was proved in~\cite{GH1}.

\begin{thm}\label{sharper}
There exist constants $C > 0$ and $c > 0$ such that
$$\frac{\pi^2}{18 \log n} \, - \, \frac{C(\log\log n)^3} {(\log n)^{3/2}} \; \le \; p_c([n]^2,2) \; \le \; \frac{\pi^2}{18 \log n} \, - \, \frac{c} {(\log n)^{3/2}}$$
for every sufficiently large $n \in \N$.
\end{thm}

Our result again corrects numerical predictions, this time for the power of $\log n$ in the second term.  Specifically, in work reported in \cite{ALev}, Stauffer interpolated between simulations and the rigorous result \eqref{h} to obtain the prediction
$$p_c([n]^2,2) \approx \frac{\pi^2}{18 \log n} \, - \, \frac{0.45} {(\log n)^{1.2}},$$
which is inconsistent with the lower bound in Theorem \ref{sharper} in the limit $n\to\infty$. Since $3/2>1.2$, the true asymptotic approach to the first approximation $\pi^2/(18\log n)$ is {\em faster} than in the above numerical prediction. Furthermore, in Section \ref{Qsec} we shall discuss how the proof of Theorem~\ref{sharper} can be adapted to a variant of bootstrap percolation called the Frob\"ose model. In this case, the resulting power $3/2$ of $\log n$ corrects the prediction $1.388$ made in \cite{GLBD} on the basis of computer calculations for the local Frob\"ose model.

The proof of Theorem~\ref{sharper} will use many of the tools and techniques of~\cite{Hol}, together with some of the ideas of~\cite{GH2}, and some new ideas. In particular, we shall bound the probability of percolation by the expected number of `good' and `satisfied' hierarchies (see Lemma~\ref{basic}, below). We will define a hierarchy as in~\cite{Hol} (see Section~\ref{hiersec}), except that our hierarchies will be much finer, each step being of order $1/\sqrt{p}$, instead of $1/p$. This means that we will have far too many hierarchies; however, almost all of these have many `large' seeds, and we shall show that these contribute a negligible amount to the sum. In order to do so, we shall need a better bound on the probability that a seed is internally spanned than the straightforward bound that sufficed in~\cite{Hol}. Fortunately, the bound we need follows easily from the simple (folklore) fact that a spanning set for a rectangle $R$ must contain no fewer than $\phi(R)/2$ elements, where $\phi(R)$ denotes the semi-perimeter of $R$ (see Lemmas~\ref{seeds} and~\ref{ex}).  Surprisingly, it appears that our proof does not extend directly to the ``modified'' bootstrap percolation model; it is the analogous bound for seeds that is missing in this case (see Section \ref{Qsec} for more information).

We finish this section by making a few definitions which we shall use throughout the proof. First, we say a set $S$ is \emph{spanned} by the set $A$ if $S \subset [A]$, and that $S$ is \emph{internally spanned} by $A$ if $S \subset [A \cap S]$. An \emph{event} is simply a collection $\A$ of subsets of $[n]^2$; we say $\A$ holds if $A \in \A$. In order to keep our formulae relatively compact, we shall sometimes write $I(S)$ for the event that $S$ is internally spanned by $A$.

Next, define two functions, $\beta$ and $g$, by
$$\beta(u) \; := \; \ds\frac{u + \sqrt{u(4-3u)}}{2} \hspace{0.5cm} \textup{and} \hspace{0.5cm} g(z) \; := \; - \log\left( \beta\left( 1 - e^{-z} \right) \right).$$
We remark that $\beta$ is increasing on $[0,1]$, and so $g$ is decreasing on $(0,\infty)$, and that $g(z) \le 2e^{-z}$ when $z$ is large (see Proposition~3 of~\cite{d=r=3}). Note that $\beta(u) \sim \sqrt{u}$ as $u \to 0$, and so $g(z) \sim - \log \sqrt{z}$ as $z \to 0$, where $g(z) \sim h(z)$ means that $g(z)/h(z) \to 1$.

A \emph{rectangle} is a set of the form
$$R=[(a,b),(c,d)] \; := \; \big\{ (x,y) \,:\, a \le x \le c, \, b \le y \le d \big\}\subset\ZZ^2,$$
where $a,b,c,d \in \ZZ$. The dimensions of $R$ are $\dim(R) = (c-a+1,d-b+1)$, the long and short side-lengths of $R$ are respectively $\sht(R) = \min\{c-a+1,d-b+1\}$ and $\lg(R) = \max\{c-a+1,d-b+1\}$, and the semi-perimeter of $R$ is $\phi(R)=\sht(R)+\lg(R)$.

We say that a rectangle $R = [(a,b),(c,d)]$ is \emph{crossed from left-to-right} by $A \subset R$ if
$$R \; \subset \; \left[ A \cup \big\{(x,y) \in \ZZ^2 \, : \, x \le a-1 \big\} \right],$$
i.e., if $R$ is spanned by $A$ together with the set of all sites to the left of $R$. Note that this is equivalent to there being no `double gap' (i.e., no adjacent `empty' pair of columns) in $R$, and the final column being occupied. (Here `empty' means `contains no element of $A$' and `occupied' means `not empty'.)

For each $p \in (0,1)$, let $q = -\log(1 - p)$, so that $p \sim q$ as $p \to 0$. To motivate this definition (and the definition of $g(z)$, above), note (from Lemma~8 of~\cite{Hol}) that for any rectangle $R$ with dimensions $(a,b)$, then
$$\Pr_p(A \textup{ crosses }R\textup{ from left-to-right}) \; \le \; e^{-ag(bq)}.$$

We shall use the notation $f(\mathbf{x} ) = O\big( h(\mathbf{x}) \big)$ throughout to mean that there exists an absolute constant $C > 0$, independent of all other variables (unless otherwise stated), such that $f(\mathbf{x}) \le C h( \mathbf{x} )$ for all $\mathbf{x} = (x_1,\ldots,x_k)$. If the constant $C$ depends on some other parameter $y$, then we shall write $f(\mathbf{x}) = O_y\big( h(\mathbf{x}) \big)$.

We shall write $\RR_+$ for the positive reals, and if $\a,\b \in \RR^2$, then we shall write $\a \le \b$ if $a_1 \le b_1$ and $a_2 \le b_2$. Thus a path in $\RR^2$ is `increasing' if it is increasing in both coordinates. Finally, if $G$ is a directed graph, then $N_G^\->(v)$ will denote the set of out-neighbours of a vertex $v$ in $G$.

The rest of the paper is organised as follows. In Section~\ref{seedsec} we give an upper bound on the probability that a sufficiently small rectangle (a seed) is internally spanned. In Section~\ref{hiersec} we recall from~\cite{Hol} the notion of a hierarchy, which is fundamental to the proof of Theorem~\ref{sharper}, together with some important lemmas from~\cite{GH2} and~\cite{Hol}. In Section~\ref{proofsec} we prove Theorem~\ref{sharper}, and in Section~\ref{Qsec} we mention some open questions.

\section{A lemma on seeds}\label{seedsec}

In this section we shall prove the following lemma, which bounds the probability that a small rectangle is internally spanned. Recall that $q = -\log(1 - p)$.

\begin{lemma}\label{seeds}
There exists $\delta > 0$ such that, for any $p > 0$, and any rectangle $R$ with $\dim(R)=(a,b)$, where $a \le b$ and $ap \le \delta$ then
$$\Pr_p\big( [A \cap R] = R \big) \; \le \; 3^{\phi(R)} \exp\Big( -  \phi(R) g(aq)  \Big).$$
\end{lemma}

We begin by recalling a lovely and well-known exercise for high school students (see~\cite{CTM} or~\cite{Wink}, for example).
Lemma~\ref{seeds} follows from it almost immediately.

\begin{lemma}\label{ex}
If $R$ is a rectangle, and $A$ internally spans $R$, then $|A \cap R| \ge \phi(R)/2$.
\end{lemma}

We also make a simple observation.

\begin{obs}\label{gobs}
If $z > 0$ is sufficiently small then
$$\log(1/\sqrt{z}) \,-\, \sqrt{z} \; \le \; g(z) \; \le \; \log(1/\sqrt{z}) \,+\, z.$$
\end{obs}

\begin{proof}
We use the estimates $z - z^2 \le 1 - e^{-z} \le z$, and $\sqrt{u} \le \beta(u) \le \sqrt{u} + u$, which are valid for small $z$ and $u$. It follows that
$$g(z) \; \ge \; -\log \beta(z) \; \ge \; -\log(\sqrt{z} + z) \; = \; -\log \sqrt{z} - \log(1 + \sqrt{z}) \; \ge \; -\log \sqrt{z} - \sqrt{z}.$$
The proof of the upper bound is similar.
\end{proof}

We can now easily deduce Lemma~\ref{seeds}.

\begin{proof}[Proof of Lemma~\ref{seeds}]
Let $m = |A \cap R|$. By Lemma~\ref{ex}, if $A$ internally spans $R$ then $m \ge (a + b)/2$. There are at most ${{ab} \choose m}$ ways to choose the set $A \cap R$, given $m$, and each occurs with probability at most $p^m$. Thus, by the union bound,
$$\Pr_p\big( [A \cap R] = R \big) \; \le \; \sum_{m \ge (a+b)/2} {{ab} \choose m} p^m \; \le \; (6aq)^{(a+b)/2}.$$
The final inequality follows since ${{ab} \choose m}p^m \le \big( \frac{eabp}{m} \big)^m \le (6aq)^m/2$ if $\delta > 0$ is sufficiently small, and since $6aq \le 12\delta < 1/2$. In the second inequality we used $p \sim q$ and $m \ge b/2$.

But $\log(1/\sqrt{aq}) \ge g(aq) - aq$, by Observation~\ref{gobs}, so
$$(aq)^{(a+b)/2} \; \le \; \exp\Big( - (a + b) g(aq) \,+\, (a+b)aq \Big).$$
The result now follows, since $aq \le 2\delta$, and $\sqrt{6}e^{2\delta} < 3$ if $\delta$ is sufficiently small.
\end{proof}

\section{Hierarchies}\label{hiersec}

In this section we shall recall some important definitions and lemmas from~\cite{GH2} and~\cite{Hol}; for the proofs, we refer the reader to those papers. In particular, we define a hierarchy as in Section~9 of~\cite{Hol}.

\begin{defn}
A \emph{hierarchy} $\HH$ for a rectangle $R \subset [n]^2$ is an oriented rooted tree $G_\HH$, with all edges oriented away from the root (`downwards'), together with a collection of rectangles $(R_u \subset [n]^2 \,:\, u \in V(G_\HH))$, one for each vertex of $G_\HH$,  satisfying the following criteria.
\begin{enumerate}
\item[$(a)$] The root of $G_\HH$ corresponds to $R$.
\item[$(b)$] Each vertex has at most $2$ neighbours below it.
\item[$(c)$] If $v \in N_{G_\HH}^\->(u)$ then $R_u \supset R_v$.
\item[$(d)$] If $N_{G_\HH}^\->(u) = \{v,w\}$ then $[ R_v \cup R_w ] = R_u$.
\end{enumerate}
\end{defn}

A vertex $u$ with $N_{G_\HH}^\->(u)=\emptyset$ is called a \emph{seed}. Given two rectangles $S \subset R$, we write $D(S,R)$ for the event (depending on the set $(A \cap R) \setminus S$) that
$$R \: = \: [(A \cup S) \cap R],$$
i.e., the event that $R$ is internally spanned by $A \cup S$.

An event $\A$ is increasing if $A \in \A$ and $A \subset A'$ implies that $A' \in \A$. Two increasing events $\B$ and $\C$ are said to \emph{occur disjointly} if there exist disjoint sets $B \subset A$ and $C \subset A$ with $B \in \B$ and $C \in \C$. We write $\B \circ \C$ for the collection of such sets $A$, i.e., the event that $\B$ and $\C$ occur disjointly. We say a hierarchy \emph{occurs} (or is \emph{satisfied} by the set $A$) if the following events all occur disjointly.
\begin{enumerate}
\item[$(e)$] For each seed $u$: $R_u$ is internally spanned by $A$.
\item[$(f)$] For each pair $(u,v)$ satisfying $N_{G_\HH}^\->(u) = \{v\}$: $D(R_v,R_u)$ occurs.
\end{enumerate}

Given two rectangles $S \subset R$, with dimensions $(a_1,a_2)$ and $(b_1,b_2)$ respectively, define
$$d_j(S,R) \; := \; \frac{b_j - a_j}{b_j}$$
for $j = 1,2$, and let $d(S,R) = \max\{d_1(S,R),d_2(S,R)\}$.

The following definition is slightly different to that in~\cite{Hol}, and is motivated by the method of~\cite{GH2} (see also Lemma~\ref{crossing} below). This definition is necessary because in order to prove a sharper result, we need to take a finer hierarchy. In our application we shall take $T = \sqrt{q}$ and $Z =\log^3(1/q)/\sqrt{q}$.

\begin{defn}
A hierarchy is \emph{good} for $(T,Z) \in \RR_+^2$ if is satisfies the following.
\begin{enumerate}
\item[$(g)$] If $N_{G_\HH}^\->(u) = \{v\}$ and $|N_{G_\HH}^\->(v)| = 1$ then $T \le d(R_v,R_u) \le 2T$.
\item[$(h)$] If $N_{G_\HH}^\->(u) = \{v\}$ and $|N_{G_\HH}^\->(v)| \neq 1$ then $d(R_v,R_u) \le 2T$.
\item[$(i)$] If $|N_{G_\HH}^\->(u)| = 2$ and $v \in N_{G_\HH}^\->(u)$, then $d(R_v,R_u) \ge T$.
\item[$(j)$] $u$ is a seed if, and only if, $\sht(R_u) \le Z$.
\end{enumerate}
\end{defn}

Before continuing, we make a simple observation about the height, $h(\HH)$ of a hierarchy $\HH$, by which we mean the maximum distance in $G_\HH$ of a leaf from the root.

\begin{lemma}\label{height}
Let $R$ be a rectangle, let $Z > 1 > T > 0$, and let $\HH$ be a hierarchy for $R$ which is good for $(T,Z)$. Then
$$h(\HH) \; \le \; \frac{8}{T} \log\left( \frac{\phi(R)}{Z} \right) \,+\, 1.$$
\end{lemma}

\begin{proof}
Consider a path $P$ of length $h(\HH)$ from the root to a leaf $u$. Let $w$ be the parent (i.e., the neighbour) of $u$ in $G_\HH$, and note that $\sht(R_w) > Z$. Note also that in every two steps backwards along $P$, at least one of the dimensions of the corresponding rectangle increases by a factor of at least $1 + T$. Hence one of the dimensions goes up by this factor at least $(h(\HH) - 1)/4$ times (on the path from $w$ to the root), and so
$$Z(1 + T)^{(h(\HH)-1)/4} \; \le \; \phi(R).$$
The result follows by rearranging and using the inequality $\log(1 + T) \ge T / 2$, which is valid for all $T \in (0,1)$.
\end{proof}

The following key lemma about hierarchies was proved in \cite{Hol}. Although our definition of hierarchy is slightly different, the proof in our case is almost identical.

\begin{lemma}[Proposition~32 of~\cite{Hol}]\label{exists}
Let $Z  > 1 > T > 0$, let $R$ be a rectangle, and suppose $A$ internally spans $R$. Then there exists a hierarchy $\HH$ for $R$, which is good for $(T,Z)$, and which is satisfied by $A$.
\end{lemma}

\begin{proof}[Sketch of proof]
We use induction on $\phi(R)$; if $\phi(R) \le Z$ then the result is trivial. Now, assume $\phi(R) > Z$ and apply Proposition~30 of~\cite{Hol} (see also~\cite{Jozsi} or~\cite{BB}) repeatedly, each time choosing the rectangle $S$ which minimizes $d(S,R)$. We stop when we obtain a rectangle $S$ such that either $\phi(S) \le Z$, or $d(S,R) \ge T$.

There are three cases. If $\phi(S) \le Z$ and $d(S,R) \le 2T$, then $\HH$ has two vertices. If $\phi(S) > Z$ and $d(S,R) \le 2T$, the root of $\HH$ has degree one, and the rest of $\HH$ can be found by applying the induction hypothesis to $S$.

So assume that $d(S,R) > 2T$, and consider the last application of Proposition~30 of~\cite{Hol}. We deduce that there exist rectangles $S'$ and $U$ with $[S \cup S'] = U$, with $d(U,R) \le T$, $\phi(U) > Z$ and $d(S',R) > 2T$, and such that $S$ and $S'$ are disjointly internally spanned by $A$. But $d(S,U) \ge d(S,R) - d(U,R) \ge T$, and similarly for $S'$. Thus, applying the induction hypothesis to $S$ and $S'$, we obtain a hierarchy $\HH$ as required.
\end{proof}

Finally, recall the following fundamental lemma of van den Berg and Kesten~\cite{vBK}.

\begin{vBKlemma}
Let $\A$ and $\B$ be any two increasing events, and let $p \in (0,1)$. Then
$$\Pr_p(\A \circ \B) \; \le \; \Pr_p(\A)\,\Pr_p(\B).$$
\end{vBKlemma}

We can now easily deduce, as in Section~10 of~\cite{Hol}, our basic bound on the probability of percolation. Given a rectangle $R$ and a pair $(T,Z) \in \RR^2$, we write $\HH(R,T,Z)$ for the collection of hierarchies for $R$ which are good for $(T,Z)$.

Recall that $\Pr_p\big( I(R) \big)$ and $\Pr_p\big( D(S,R) \big)$ denote the probabilities in $\Pr_p$ of the events ``$R$ is internally spanned by $A$" and ``$R$ is internally spanned by $A \cup S$" respectively.

\begin{lemma}\label{basic}
Let $R$ be a rectangle in $[n]^2$, let $Z > 1 > T > 0$, and let $p > 0$. Then
$$\Pr_p\Big([A \cap R] = R \Big) \; \le \; \sum_{\HH \in \HH(R,T,Z)} \Bigg( \prod_{N_{G_\HH}^\->(u) = \{v\}} \Pr_p\Big( D(R_v,R_u) \Big) \Bigg) \left(\prod_{\textup{ seeds }u} \Pr_p\Big( I(R_u) \Big) \right).$$
\end{lemma}

(Above and in subsequent usage, the first product is over all pairs of vertices $(u,v)$ of $\HH$ that satisfy the given condition $N_{G_\HH}^\->(u) = \{v\}$, and the second product is over all seeds $u$ of $\HH$.)

\begin{proof}[Proof of Lemma \ref{basic}]
By Lemma~\ref{exists}, if $A$ internally spans $R$ then there exists a hierarchy in $\HH(R,T,Z)$ which is satisfied by $A$. Hence the probability that $A$ internally spans $R$ is bounded above by the expected number of such hierarchies. Since  the events $I(R_u)$ and $D(R_v,R_u)$ are all monotone, and all occur disjointly (see $(e)$ and $(f)$ above), the result follows by the van den Berg-Kesten Lemma.
\end{proof}

We recall the following lemma of Aizenman and Lebowitz~\cite{AL}, which is a standard tool for proving lower bounds on $p_c([n]^d,2)$.

\begin{lemma}\label{k2k}
Suppose $A$ internally spans $[n]^2$. Then, for all $1 \le L \le n$, there exists a rectangle $R$, internally spanned by $A$, with
$$L \; \le \; \lg (R) \; \le \; 2L.$$
\end{lemma}

We recall also the following bound on $\Pr_p\big( D(R,R') \big)$ from~\cite{GH2}.

\begin{lemma}[Lemma~5 of~\cite{GH2}] \label{crossing}
Let $R \subset R'$ be rectangles of dimensions $(a, b)$ and $(a+s,b+t)$ respectively, and let $p > 0$. Then
$$\Pr_p\big( D(R,R') \big) \; \le \;  \exp \Big( - sg(bq) - tg(aq) + 2\big( g(bq) + g(aq) \big)+ (qst)e^{2g(bq) + 2g(aq)} \Big).$$
\end{lemma}

The following observation follows exactly as in Lemma~10 of~\cite{GH2}.

\begin{obs}[Lemma~10 of~\cite{GH2}]\label{e^g}
Let $B > 0$ be sufficiently large, and let $a \in \N$ and $q > 0$ satisfy $a \le B/q$. Then
$$e^{2g(aq)} \; \le \; \ds\frac{4B}{aq}.$$
\end{obs}

\begin{proof}
Let $z > 0$ and $u = 1 - e^{-z}$, and recall that $e^{-g(z)} = \beta(u)$. Recall also that $\beta(u) \sim \sqrt{z}$ when $z \to 0$ and that $\beta(u) \to 1$ as $z \to \infty$. Thus, since $B > 0$ is sufficiently large, it follows that
$$\beta(u) \; \ge \; \frac{1}{2} \sqrt{\frac{z}{B}}$$
for every $z \le B$, as required.
\end{proof}

We shall need a couple more definitions in order to rewrite Lemmas~\ref{basic} and~\ref{crossing} in a more useful form. Given $\a,\b \in \RR_+^2$ with $\a \le \b$, let
$$W_g(\mathbf{a},\mathbf{b}) \; := \; \ds\inf_{\gamma \,:\, \mathbf{a} \to \mathbf{b}} \int_\gamma \Big( g(y) \,dx \:+\: g(x) \,dy \Big),$$
where the infimum is taken over all piecewise linear, increasing paths from $\mathbf{a}$ to $\mathbf{b}$ in $\RR^2$ (see Section 6 of \cite{Hol}). Now, for any two rectangles $R \subset R'$, and given $p > 0$, define
$$U(R,R') \; = \; W_g\big( q \dim(R), q \dim(R') \big).$$
The following observation is immediate from the definition.

\begin{obs}[Proposition~13 of~\cite{Hol}] \label{triv}
Let $R \subset R'$ be rectangles of dimensions $(a, b)$ and $(a+s,b+t)$ respectively, and let $p > 0$. Then
$$sg(bq) + tg(aq) \; \ge \; \frac{1}{q} U(R,R').$$
\end{obs}

Let $N(\HH)$ denote the number of vertices in a hierarchy $\HH$, and $M(\HH)$ denote the number of vertices of $\HH$ which have outdegree two. The following technical lemma was proved in \cite{Hol}. Again, although our definition is slightly different, the proof is identical.

\begin{lemma}[Lemma~37 of~\cite{Hol}]\label{Holpod}
Let $T,Z \in \RR_+$, let $\HH$ be a hierarchy for the rectangle $R$ which is good for the pair $(T,Z)$, and let $p > 0$. There exists a rectangle $S \subset R$, called the `pod' of $\HH$,  such that
$$\dim(S) \; \le \; \sum_{\textup{seeds }u} \dim(R_u)$$
and $$\sum_{N_{G_\HH}^\->(v) = \{w\}} U(R_w,R_v) \; \ge \; U(S,R) \,-\, 2q M(\HH) g(Zq).$$
\end{lemma}

We remark that although the rectangle $S$ is not necessarily unique, Lemma~\ref{Holpod} allows us to select such a rectangle $S(\HH)$ for each good hierarchy $\HH \in \HH(R,T,Z)$. We shall refer to this rectangle as `the pod of $\HH$'.

We shall use the following observation to bound $U(S,R)$ from below, and again later in the proof of Theorem~\ref{sharper}.

\begin{obs}\label{intobs}
There exists $C > 0$ such that, for every $0 < a < \infty$, we have
$$\int_0^a g(z) \, dz \; \le \; \frac{a}{2} \log\left( 1 + \frac{1}{a} \right) \,+\, Ca.$$
\end{obs}

\begin{proof}
Let $\eps > 0$ be such that Observation~\ref{gobs} holds when $z \le \eps$. Then, if $a \le \eps$ we have
$$\int_0^a g(z) \, dz \; \le \;  \frac{1}{2} \int_0^a -\log z + 2z \, dz \; \le \; \frac{a}{2} \log \frac{1}{a} \,+\, a \,+\, a^2,$$
as required. Moreover, since $g$ is decreasing, we have
$$\int_\eps^a g(z) \, dz \; \le \; a g(\eps),$$
and so the observation follows, since if $a > \eps$ then $\int_0^a g(z) \, dz \le 1 + ag(\eps) = O(a)$.
\end{proof}

Finally, we shall use the following lemma, which follows from Lemma~16 of~\cite{Hol} (see also Lemma~7 of~\cite{GH2}).

\begin{lemma}\label{USR}
Let $q > 0$ and $S \subset R$, with $\dim(S) = (a,b)$ and $\dim(R) = (A,B)$, where $A \le B$. If $b \le A$, then
$$\frac{1}{q} U(S,R) \; \ge \; \frac{2}{q} \, \int_0^{Aq} g(z) \, dz \, + \, \big(B-A\big) g(Aq) \, - \, \frac{\phi(S)}{2} \log \left( 1 + \frac{1}{\phi(S)q} \right) \, - \, O\big( \phi(S) \big).$$
If $b > A$, then
$$\frac{1}{q} U(S,R) \; \ge \; (A - a)g(bq) \, + \, \big(B-b\big) g(Aq).$$
\end{lemma}

\begin{proof}
Assume first that $b \le A$. By Lemma~16 of~\cite{Hol}, the path integral is minimized by paths which follow the main diagonal as closely as possible. Assuming for simplicity that $a \le b$, by following the piecewise linear path $(aq,bq) \to (bq,bq) \to (Aq,Aq) \to (Aq,Bq)$ we obtain
$$\frac{1}{q} U(S,R) \; \ge \; (b - a)g(bq) \,+\,\frac{2}{q} \, \int_{bq}^{Aq} g(z) \, dz \, + \, \big(B-A\big) g(Aq).$$
Now, by Observation~\ref{intobs}, we have
$$\frac{2}{q} \int_0^{bq} g(z) \, dz \; \le \; b \log \left( 1 + \frac{1}{bq} \right) \,+\, O(b),$$
and by Observation~\ref{gobs} we have $g(bq) \ge \frac{1}{2}\log(1 + 1/bq) - O(1)$. (Note that inequality is trivial if $bq$ is not sufficiently small.) Hence
$$(b-a)g(bq) \,-\, \frac{2}{q} \int_0^{bq} g(z) \, dz \; \ge \; - \frac{a+b}{2} \log \left( 1 + \frac{1}{bq} \right)  \, - \, O(b),$$
as required. The inequality for $b > A$ can be obtained by following the path $(aq,bq) \to (Aq,bq) \to (Aq,Bq)$, and applying Lemma~16 of~\cite{Hol}.
\end{proof}

\section{The proof of Theorem~\ref{sharper}}\label{proofsec}

In this section we shall put together the pieces and prove Theorem~\ref{sharper}. Recall that, given $p > 0$, we define $q = -\log(1 - p) \sim p$ as $p \to 0$.

\begin{prop}\label{integral}
Let $C > 0$ and $\eps > 0$ be constants, let $p = p(C,\eps) > 0$ be sufficiently small, and let $R$ be a rectangle with dimensions $(a,b)$, where
$$\frac{\eps}{q} \; \le \; a \; \le \; b \; \le \; \ds\frac{C}{q} \log \left( \frac{1}{q} \right).$$
Then
$$\Pr_p\Big( [A \cap R] = R\Big) \; \le \; \exp\left( - \left[ \frac{2}{q} \int_0^{aq} g(z) dz \,+\, (b - a)g(aq) \right] \,+\, \frac{O_C(1)}{\sqrt{q}} \left( \log \frac{1}{q} \right)^3 \right).$$
\end{prop}

We remark that the constant implicit in the $O_C(1)$ term depends on the constant $C$, but not on the variables $p$, $a$ and $b$ (and also not on the constant $\eps$).

We begin by defining some of the parameters we shall use. First, set $B = C\log(1/q)$, so that $a \le b \le B/q$, set $T = \sqrt{q}$, and set
$$Z \; = \; \frac{1}{ \sqrt{q} } \left( \log \frac{1}{q} \right)^3.$$
Let $S = S(\HH)$ denote the pod of a hierarchy $\HH$, given by Lemma~\ref{Holpod}.

\begin{lemma}\label{claim}
Let $C,\eps,p > 0$, $a,b \in \N$ and the rectangle $R$ be as in the statement of Proposition~\ref{integral}, and let $B$, $T$ and $Z$ be as defined above. Then
$$\Pr_p\big( I(R) \big) \; \le \; \sum_{\HH \in \HH(R,T,Z)} \exp \left[ - \frac{1}{q} U(S,R) \, + \, O_C\left( N(\HH) \left( \log\frac{1}{q} \right)^2 \right) \right] \prod_{\textup{ seeds }u} \Pr_p\big( I(R_u) \big).$$
\end{lemma}

\begin{proof}
First note that by Observation~\ref{triv} and Lemma~\ref{Holpod}, the pod $S = S(\HH) \subset R$ of $\HH$ satisfies
$$\sum_{N_{G_\HH}^\->(u_i) = \{v_i\}} s_ig(b_iq) + t_ig(a_iq) \; \ge \; \frac{1}{q} \sum_{N_{G_\HH}^\->(u) = \{v\}} U(R_v,R_u) \; \ge \; \frac{1}{q} U(S,R) \,-\, 2M(\HH) g(Zq),$$
where $(a_i,b_i)$ and $(a_i+s_i,b_i+t_i)$ are the dimensions of $R_{v_i}$ and $R_{u_i}$ respectively.

Now, by the definition of a hierarchy, we have $s_i \le 2T(a_i+s_i) \le 3Ta_i$, and similarly $t_i \le 3Tb_i$, for every pair $(u_i,v_i)$ with $N_{G_\HH}^\->(u_i) = \{v_i\}$. Recall that $g(z)$ is decreasing, so
$$\max\Big\{ g(Zq),g(a_iq),g(b_iq) \Big\} \; \le \; g(q) \; \le \; \log \frac{1}{q},$$
by Observation~\ref{gobs} (applied with $z = q$). Recall also that $a_i,b_i \le b \le B/q$.

By Observation~\ref{e^g}, it follows that $s_i e^{2g(a_iq)} \le 4Bs_i/a_iq \le 12BT/q$, and similarly $t_i e^{2g(b_iq)} \le 12BT/q$. Thus
$$g(Zq) \, + \, 2 g(a_iq) \, + \, 2 g(b_iq) \, + \, (qs_it_i)e^{2g(a_iq) + 2g(b_iq)} \; \le \; 5 \log\frac{1}{q} \, + \, O\left( \frac{B^2T^2}{q} \right),$$
and hence, since $T^2 = q$, $B = O_C(\log(1/q))$ and $M(\HH) \le N(\HH)$,
$$2M(\HH) g(Zq) \, + \, \sum_{N_{G_\HH}^\->(u) = \{v\}} \Big( 2g(bq) \,+\, 2g(aq) \,+\, (qst)e^{2g(bq) + 2g(aq)}  \Big) \; = \; O_C\left( N(\HH) \left( \log\frac{1}{q} \right)^2 \right).$$
Hence, by Lemma~\ref{crossing}, we have
$$\prod_{N_{G_\HH}^\->(u) = \{v\}} \Pr_p\big( D(R_v,R_u) \big) \; \le \; \exp \left[ - \frac{1}{q} U(S,R) \, + \, O_C\left( N(\HH) \left( \log\frac{1}{q} \right)^2 \right) \right],$$
and so the lemma follows by Lemma~\ref{basic}.
\end{proof}

We can now deduce Proposition~\ref{integral} from Lemma~\ref{claim}. The main difficulty lies in the fact that there are too many hierarchies: there could be as many as $2^{1/\sqrt{q}}$ vertices in $G_\HH$, and for each vertex $u$ we have many choices for the rectangle $R_u$. However, most of these hierarchies have many seeds, and those with many \emph{large} seeds have rather small weight in the sum. This turns out to be the key idea in the proof.

Indeed, given a hierarchy of $R$ which is good for $(T,Z)$, define a \emph{large seed} to be one with $\phi(R_u) \ge Z/3$. We make the following key observation.

\begin{obs}
Let $\HH \in \HH(R,T,Z)$, and assume that $4T \le 1$ and $Z \ge 6$. Then every vertex of $\HH$ is either a seed, or lies above at least one large seed.
\end{obs}

\begin{proof}
By the definition of a good hierarchy, either $u$ is a seed, or $\phi(R_u) > Z$ and either $N_{G_\HH}^\->(u) = \{v\}$ or $N_{G_\HH}^\->(u) = \{v,w\}$. In the former case we have
$$\phi(R_v) \; \ge \; \big( 1 - 2T \big) \phi(R_u) \; \ge \; \frac{Z}{2},$$
since $4T \le 1$. In the latter case, we have $\phi(R_v) + \phi(R_w) \ge \phi(R_u) - 2$, and so
$$\max\big\{\phi(R_v),\phi(R_w)\big\} \; \ge \; \frac{Z}{3},$$ as required.
\end{proof}

Let the number of large seeds in a hierarchy $\HH$ be denoted $m(\HH)$.

\begin{proof}[Proof of Proposition~\ref{integral}]
Let $R \subset [n]^2$, $p > 0$, $B = C\log(1/q)$, $T = \sqrt{q}$, and $Z = (1/\sqrt{q}) \big( \log \frac{1}{q} \big)^3$ be as described above, and suppose that $\HH$ is a hierarchy for $R$ which is good for the pair $(T,Z)$. Recall that $p > 0$ is chosen sufficiently small, and that $a \le b \le B/q$. By Lemma~\ref{height}, $\HH$ has height at most $(10 / \sqrt{q}) \log (1/q)$, and hence the number of vertices $N(\HH)$ in $G_\HH$ satisfies
\begin{equation}\label{eqN}
N(\HH) \; \le \; 2m(\HH) \cdot h(\HH) \; = \; O\left( \frac{m(\HH)}{\sqrt{q}} \log \frac{1}{q} \right).
\end{equation}
Therefore, the number of hierarchies with $m$ large seeds (i.e., with $m(\HH) = m$) is at most
\begin{equation}\label{eqcount}
\sum_N \left( \frac{B}{q} \right)^{4N} \; \le \; \exp\left( O(1) \frac{m}{\sqrt{q}} \left( \log\frac{1}{q} \right)^2 \right).
\end{equation}

Now, for each hierarchy $\HH$, define
$$X(\HH) \; := \; \sum_{\textup{seeds }u} \phi(R_u),$$
and note that $X(\HH) \ge \ds\frac{m(\HH)Z}{3}$, and that $\phi\big( S(\HH) \big) \le X(\HH)$, by Lemma~\ref{Holpod}. By Lemma~\ref{seeds}, for every seed $R_u$ we have
$$\Pr_p\big( I(R_u) \big) \; \le \; 3^{\phi(R_u)} \exp\Big( - \phi(R_u) g(Zq) \Big),$$
since $\sht(R_u) \le Z =o(1/q)$ as $q \to 0$, and $g(z)$ is decreasing in $z$. Thus
\begin{equation}\label{eqseeds}
\prod_{\textup{ seeds }u} \Pr_p\big( I(R_u) \big) \; \le \; 3^{X(\HH)} \exp\Big( -X(\HH) g(Zq) \Big).
\end{equation}
We split into two cases. The first is easier to handle, and we shall not have to approximate too carefully; in the second the calculation is much tighter.

\bigskip
\noindent \textbf{Case 1}: $\lg(S) > a$.

\medskip
We have, by Lemma~\ref{claim} combined with~\eqref{eqN} and~\eqref{eqseeds},
$$\Pr_p\big( I(R) \big) \; \le \; \sum_{\HH \in \HH(R,T,Z)} 3^{X(\HH)} \exp \Bigg[ - \,\frac{1}{q} U(S,R) \,-\, X(\HH) g(Zq) \, + \, O_C\left( \frac{m(\HH)}{\sqrt{q}} \left( \log\frac{1}{q} \right)^3 \right) \Bigg].$$
Recall that $a < \phi(S) \le X(\HH)$, by Lemma~\ref{Holpod}, and so $\ds\frac{1}{q} U(S,R) \ge \big( b - X(\HH) \big)g(aq)$, by Lemma~\ref{USR}. Hence the summand above is at most
$$3^{X(\HH)} \exp \Bigg[ - b g(aq)  \,-\, X(\HH) \Big( g(Zq) - g(aq) \Big) \, + \, O_C\left( \frac{m(\HH)}{\sqrt{q}} \left( \log\frac{1}{q} \right)^3 \right) \Bigg].$$
Now, since $g$ is decreasing, $X(\HH) \ge m(\HH)Z/3$ and $a/Z \ge q^{-1/3}$, by Observation~\ref{gobs} we have
$$X(\HH)\big( g(Zq) - g(aq) \big) \; \ge \; \frac{X(\HH)}{7} \log\left( \frac{1}{q} \right) \; = \; \frac{1}{o(1)} \,m(\HH)Z \; = \; \frac{1}{o(1)}\, \frac{m(\HH)}{\sqrt{q}} \left( \log \frac{1}{q} \right)^3$$
as $q \to 0$. It follows that
$$\Pr_p\big( I(R) \big) \; \le \; \sum_{\HH} \exp \left( - b g(aq) \, - \, \frac{X(\HH)}{8} \log\frac{1}{q} \right)\; \le \; \exp\left( - \frac{2}{q} \int_0^{aq} g(z) \, dz \, - \, b g(aq) \right),$$
as required. Indeed, we showed that $X(\HH) \big( g(Zq) - g(aq) \big)$ is at least $\frac{X(\HH)}{7} \log \frac{1}{q}$, and much bigger than $\frac{m(\HH)}{\sqrt{q}} \big( \log\frac{1}{q} \big)^3$, so the first inequality holds. For the last inequality, first note that
$$\frac{2}{q} \int_0^{aq} g(z) \, dz \; \le \; a \log\left( 1 + \frac{1}{aq} \right) \,+\, O(a) \; = \; o\left( X(\HH) \log \frac{1}{q}\right)$$
as $q \to 0$, by Observation~\ref{intobs}. Here we used the facts that $aq \ge \eps$ and $a < X(\HH)$. Finally, note that, by~\eqref{eqcount}, there are at most $e^{X}$ hierarchies with $X(\HH) = X$. Hence we obtain a geometrically decreasing sum over $X$, and the claimed bound follows.

\bigskip
\noindent \textbf{Case 2}: $\lg(S) \le a$.

\medskip
By Lemma~\ref{USR}, and since $\phi(S) \le X(\HH)$, we have
$$\frac{1}{q} U(S,R) \; \ge \; \frac{2}{q} \, \int_0^{aq} g(z) \, dz \, + \, \big(b-a\big) g(aq) \, - \, \frac{X(\HH)}{2} \log \left( 1 +  \frac{1}{X(\HH)q} \right) \, - \, O\big( X(\HH) \big).$$
Hence, by~\eqref{eqN},~\eqref{eqseeds} and Lemma~\ref{claim}, we have
\begin{align*}
& \Pr_p\big( I(R) \big) \, \le \sum_{\HH \in \HH(R,T,Z)} \exp\Bigg[  - \frac{2}{q} \, \int_{0}^{aq} g(z) \, dz \, - \, \big( b - a \big) g(aq) \, + \, \frac{X(\HH)}{2} \log \left( 1 + \frac{1}{X(\HH)q} \right) \\
& \hspace{6cm}  \,+\, O\big (X(\HH) \big) \, + \, O_C\left( \frac{m(\HH)}{\sqrt{q}} \left( \log\frac{1}{q} \right)^3 \right)  \,-\, X(\HH) g(Zq) \Bigg].
\end{align*}
But by Observation~\ref{gobs},
$$X(\HH) \left( \frac{1}{2} \log \left( 1 + \frac{1}{X(\HH)q} \right) \,+\, c_1 \, - \, g(Zq) \right) \; \le \; - \, \frac{X(\HH)}{2} \log \left( \frac{X(\HH)}{C_1 Z \big(1 + X(\HH)q \big)} \right),$$
where $C_1 = e^{2c_1+1}$. Note that, for any $u, v > 0$, the function $x \log\left( \frac{x}{u x + v} \right)$ is increasing if $x \ge ux + v$,  and recall that $X(\HH) \ge \ds\frac{m(\HH)Z}{3}$. Thus, the right-hand side above is decreasing in $X(\HH)$ if $m(\HH)$ is sufficiently large, and hence either $m(\HH) = O(1)$, or
$$- \, \frac{X(\HH)}{2} \log \left( \frac{X(\HH)}{C_1 Z \big(1 + X(\HH)q \big)} \right) \; \le \; - \, \frac{m(\HH)Z}{6} \log \left( \frac{m(\HH)}{4C_1} \right),$$
and $m(\HH)Zq \le 1$, or the left-hand side is at most $-\frac{m(\HH)Z}{6} \log(1/q)$.

Finally, recalling that $Z = \frac{1}{\sqrt{q}} \big( \log (1/q) \big)^3$, we have
$$ - \, \frac{m(\HH)Z}{6} \log \left( \frac{m(\HH)}{4C_1} \right) \,+\, O_C\left( \frac{m(\HH)}{\sqrt{q}} \left( \log\frac{1}{q} \right)^3 \right) \; \le \; \frac{O_C(1)}{\sqrt{q}} \left( \log\frac{1}{q} \right)^3 \,-\,  \frac{m(\HH)}{\sqrt{q}} \left( \log\frac{1}{q} \right)^3,$$
since either $m(\HH) = O_C(1)$, or the first (negative) term dominates.

Putting these various bounds together gives
\begin{align*}
& \Pr_p\big( I(R) \big) \, \le \sum_{\HH \in \HH(R,T,Z)} \exp\Bigg[  - \frac{2}{q} \, \int_{0}^{aq} g(z) \, dz \, - \, \big( b - a \big) g(aq) \, + \, \frac{O_C(1) - m(\HH)}{\sqrt{q}} \left( \log\frac{1}{q} \right)^3 \Bigg].
\end{align*}
Hence, using~\eqref{eqcount}, and summing over $m(\HH)$, we obtain
\begin{eqnarray*}
\Pr_p\big( I(R) \big) & \le & \exp\left[  - \,\frac{2}{q} \, \int_{0}^{aq} g(z) \, dz \, - \, \big( b - a \big) g(aq) \,+\, \frac{O_C(1)}{\sqrt{q}} \left( \log\frac{1}{q} \right)^3 \right],
\end{eqnarray*}
as required.
\end{proof}

Before deducing Theorem~\ref{sharper} from Proposition~\ref{integral}, we need to recall the following fact from~\cite{Hol}, and to make an easy observation.

\begin{lemma}[Proposition~5 of~\cite{Hol}]\label{pi}
$$\int_0^\infty g(z) \, dz \; = \; \frac{\pi^2}{18}.$$
\end{lemma}

The following observation follows almost immediately from Lemma~\ref{pi}.

\begin{obs}\label{final}
Let $p > 0$ be sufficiently small, and let $a,b \in \RR_+$, with $a \le b$ and $b \ge B/2p$, where $B = 10\log (1/p)$. Then
$$\frac{2}{q} \int_0^{aq} g(z) \, dz \,+\, (b - a)g(aq) \; \ge \; \frac{2\lambda}{q} \,-\, 1,$$
where $\lambda = \pi^2/18$.
\end{obs}

\begin{proof}
If $a \le B/4p$, then this follows since $\int_{aq}^\infty g(z) \, dz = O(g(aq))$, uniformly over $a \in (0,\infty)$, and so
$$(b - a) g(aq) \,-\, \frac{2}{q} \int_{aq}^\infty g(z) \, dz \; \ge \; \left( \frac{B}{4p} \right) g(aq) \,-\, O\left( \frac{g(aq)}{q}  \right) \; > \; 0.$$
If $a \ge B/4p$ then it holds because $g(z) \le 2e^{-z}$ for $z$ large, and so
$$\frac{2}{q} \int_{aq}^\infty g(z) \, dz \; \le \; \frac{4}{q} e^{-aq} \; \le \; \frac{4}{q} e^{-B/5} \; \le \; 1,$$
as required.
\end{proof}

Finally, we deduce Theorem~\ref{sharper} from Proposition~\ref{integral}.

\begin{proof}[Proof of Theorem~\ref{sharper}]
Let $C_2 > 0$ be a large constant to be chosen later, let $n \in \N$ be sufficiently large, and let
$$p \; = \; \frac{\pi^2}{18 \log n} \, - \, \frac{C_2 (\log\log n)^3}{(\log n)^{3/2}}.$$
Note that $q = -\log(1 - p) < p + p^2$, and so $q$ also satisfies this equality (with a slightly different constant $C_2$).

Let the elements of $A \subset [n]^2$ be chosen independently at random, each with probability $p$, and suppose that $A$ percolates. Then, by Lemma~\ref{k2k}, there exists a rectangle $R \subset [n]^2$, which is internally spanned by $A$, and with $B/2p \le \lg(R) \le B/p$, where $B = 10\log(1/p)$. Let $\dim(R) = (a,b)$, and assume without loss of generality that $a \le b$. There are at most $n^2 (B/p)^2$ potential such rectangles, and each is internally spanned with probability at most
$$\Pr_p\Big( [A \cap R] = R \Big) \; \le \; \exp\left( - \left[ \frac{2}{q} \int_0^{aq} g(z) dz \,+\, (b - a)g(aq) \right] \,+\, \frac{O(1)}{\sqrt{q}} \left( \log \frac{1}{q} \right)^3 \right)$$
if $\sht(R) \ge 1/q$, by Proposition~\ref{integral}, and with probability at most
$$e^{-bg(aq)} \; \le \; e^{-B/40p} \; = \; p^{1/4p} \; \le \; \left( \frac{1}{n} \right)^{100}$$
if $a = \sht(R) \le 1/q$ and $n$ is sufficiently large, since $g(aq) \ge g(1) = -\log\beta \big( \frac{e-1}{e} \big) > 1/20$ and $b \ge B/2p$. Note that, since we apply Proposition~\ref{integral} with $C = 10$, we obtain an absolute constant $O(1)$ in the expression above.

By Observation~\ref{final}, we have
$$\frac{2}{q} \int_0^{aq} g(z) dz \,+\, (b - a)g(aq) \; \ge \;  \frac{2\lambda}{q} \,-\, 1,$$
where $\lambda = \pi^2/18$. Thus, using the identity $\frac{1}{x-y} = \frac{1}{x} + \frac{y}{x(x-y)}$, this gives, as $n \to \infty$,
\begin{eqnarray*}
\Pr_p\Big( [A] = [n]^2 \Big) & \le & n^2 (B/p)^2 \exp\left( - \frac{2 \lambda}{q} \,+\, \frac{O(1)}{\sqrt{q}} \left( \log \frac{1}{q} \right)^3 \right)\\
& \le & n^2 (B/p)^2 \exp\left( - \, 2\log n \, - \, \frac{C_2}{\lambda} (\log\log n)^3\sqrt{\log n} \,+\, \frac{O(1)}{\sqrt{q}} \left( \log \frac{1}{q} \right)^3 \right) \\
& \le & n^2 (\log n)^3 \exp\bigg( - 2\log n \, - \, (\log\log n)^3\sqrt{\log n} \bigg) \; \to \; 0
\end{eqnarray*}
if $C_2$ is sufficiently large, as required.
\end{proof}

\section{Extensions and open questions}\label{Qsec}

In this paper we have studied bootstrap percolation on one particular graph, the two-dimensional grid with nearest-neighbour bonds. It is natural to ask whether our method can be applied to bootstrap percolation on other graphs; here we shall discuss two such possible generalizations.

The most obvious (and most extensively studied) generalization is to consider bootstrap percolation in $d$ dimensions (i.e., on the graph $[n]^d$), with nearest neighbour interaction and threshold $2 \le r \le d$ (as studied in, for example,~\cite{AL,d=r=3,alldr,CC,CM,Sch}). The sharp metastability thresholds for these models (with $d$ fixed, and as $n \to \infty$) were determined in~\cite{alldr}, and it is likely that the methods of this paper (and those of~\cite{GH1}) could be adapted to give improved bounds in the case of $r=2$ and general $d$.

\begin{prob}
Bound the second term in the asymptotic expansion of $p_c([n]^d,2)$ as $n\to\infty$.
\end{prob}

The case $r \ge 3$ is more complicated, and the following problem is likely to be difficult.

\begin{prob}
Bound the second term in the asymptotic expansion of $p_c([n]^3,3)$  as $n\to\infty$.
\end{prob}

We remark that in \emph{very} high dimensions ($d \gg \log n$, say) new ideas are required, and much less is known in general. However, results analogous to Theorem~\ref{sharper} have been proved in the special cases $r=2$ and $r = d$, see~\cite{Maj,cube}.

A second natural generalization is to consider bootstrap percolation in two dimensions, but with a different update rule. For example, in the `modified' bootstrap process (see~\cite{Holddim}), a vertex is infected if at least one of its neighbours in \emph{each} dimension is already infected; in the `$k$-cross' process (see~\cite{HLR,BM}), a vertex $v$ is infected if at least $k$ vertices in the cross-shaped set
$$\bigcup_{0 \neq j \in [-k+1,k-1]} \big\{ v + (0,j) , v + (j,0) \big\}$$
are previously infected; and in the Frob\"ose process (introduced by Frob\"ose~\cite{Frob} in 1989) a site of $[n]^2$ is infected if it has one already-infected neighbour in each dimension, along with the next-nearest neighbour in the corner between them. In general,
one could consider an arbitrary neighbourhood $N(v)$ of each vertex $v$, an arbitrary (monotone) family $\A(v)$ of subsets of $N(v)$, and say that $v$ becomes infected if the already-infected subset of its neighbours is in $\A(v)$.

Holroyd~\cite{Hol} (see also~\cite{Holddim}) determined the sharp threshold for the modified and Frob\"ose models, and Holroyd, Liggett and Romik~\cite{HLR} did so for the $k$-cross process for all fixed $k \in \N$. Moreover, Duminil-Copin and Holroyd~\cite{DCH} have recently shown, for a large family of such models (including all of the examples above, and other similar models), that there exists a sharp metastability threshold. It is not unreasonable to hope that our method (together with that of~\cite{GH1}) might yield improved bounds on the critical probability for a more general collection of bootstrap processes, of the type considered in~\cite{DCH}. Indeed, for two of the processes described above this is the case.

Let $p_c^{(\mathrm{F})}([n]^2)$ denote the critical probability for percolation in the Frob\"ose process on $[n]^2$, and let $p_c^{(+)}([n]^2,k)$ denote the critical probability for percolation in the $k$-cross process. The upper bounds in the following theorem were proved by Gravner and Holroyd~\cite{GH1} (for the Frob\"ose model) and by Bringmann and Mahlburg~\cite{BM} (for the $k$-cross process). The lower bounds follow by the methods of this paper.

\begin{thm}\label{frobthm}
$$p^{(\mathrm{F})}_c([n]^2) \; = \; \frac{\pi^2}{6\log n} \,-\, \frac{1}{(\log n)^{3/2+o(1)}}.$$
as $n \to \infty$. Let $k \in \N$, and let $\lambda_k = \pi^2 / 3k(k+1)$. Then
$$p^{(+)}_c([n]^2,k) \; = \; \frac{\lambda_k}{\log n} \,-\, \frac{1}{(\log n)^{3/2+o(1)}}$$
as $n \to \infty$.
\end{thm}

In fact the bounds we prove (and those from~\cite{GH1,BM}) are a little stronger than those stated above; they are like the bounds in Theorem~\ref{sharper}.

\begin{proof}[Sketch of proof of Theorem \ref{frobthm}]
For the first part, it suffices to show that (in the Frob\"ose process) on $R = [m] \times [n]$, all spanning sets have size at least $m + n - 1$. The result then follows in exactly the same way as Theorem~\ref{sharper}. Indeed, simply replace the function $g$ by the function
$$h(z) \; = \; -\log\left( 1 - e^{-z} \right),$$
and note that a rectangle is crossed if, and only if, it has no empty column. The rest of the proof carries over essentially verbatim, the key point being that $\int_0^x h(z) \, dz  \sim x \log(1/x)$ when $x \to 0$, and $\Pr_p(I(R_u)) \approx \exp\big( -\phi(R_u) h(aq) \big)$ when $u$ is a seed, so the corresponding terms in the final calculation are of the same order.

We shall give two proofs that if $[A] = R$ then $|A| \ge m + n - 1$. The first is standard, using Proposition~30 of~\cite{Hol} and induction on $\phi(R)$ (see Lemma~7 of~\cite{BB} or Problem~35 of~\cite{CTM}). For the second, consider the (bipartite) graph $G$ whose vertices are the rows and columns of $R$, with an edge from row $x$ to column $y$ if and only if $(x,y) \in A$.

To prove that $G$ has at least $m + n - 1$ edges, we shall show that it is connected. Indeed, if $G$ is not connected then exists a set of rows $X$ and a set of columns $Y$ such that $A \subset S = (X \cap Y) \cup (X^c \cap Y^c)$. But then $[S] = S \neq R$, so $A$ does not percolate, as required.

For the second part, we need the following idea from~\cite{HLR}: first couple the $k$-cross process with an `enhanced process' (see~\cite{HLR}, Section~5) in which the closed sets are rectangles. In the enhanced process the minimum number of sites required to infect an $[m] \times [n]$ rectangle is about $(m + n)/k$, which is also the typical number required. (To prove this, apply the standard proof, by induction on $m + n$.)

The result now follows by the proof of Theorem~\ref{sharper}, replacing the function $g$ by the function $-\log f(e^{-z})$, where $f : [0,1] \to [0,1]$ is decreasing and satisfies
$$f^k \,-\, f^{k+1} \; = \; x^k \,-\, x^{k+1},$$
and noting that if a rectangle is crossed in the enhanced process, then it has no `$k$-gap' of $k$ successive empty columns (see Lemma~12 of~\cite{HLR}).

We obtain a sufficiently strong bound on $\Pr_p(I(R_u))$, where $u$ is a seed, using the proof of Lemma~\ref{seeds}; this works because our lower bound on $|A \cap R|$ is also the typical size of a percolating set in $R$ when $\phi(R) \ll 1/p$. It follows that the contribution of the large seeds to the final calculation is of the same order as that of the integral $\int_0^x -\log f(e^{-z}) \, dz$, where $x$ is the semi-perimeter of the pod. Modulo a little basic analysis, the rest of the proof works as above; we leave the details to the reader.
\end{proof}

Gravner and Holroyd~\cite{GH1} also improved the upper bounds for the modified process. However, the proof of Theorem~\ref{sharper}  does not work for the modified process, since we do not have a result analogous to Lemma~\ref{seeds}. In particular, it is possible to internally span an $m \times n$ rectangle with $\max\{m,n\}$ infected sites, but the proportion of such minimal-size sets  which percolate is very small.

 Let $p_c^{(\mathrm{M})}([n]^d)$ denote the critical probability for percolation in the modified bootstrap process on the graph $[n]^d$, i.e., the infimum over $p$ such that the probability of percolation is at least $1/2$. We have the following conjecture; it is the analogue of Conjecture~\ref{GHconj} for the modified process.

\begin{conj}\label{mod}As $n\to\infty$,
$$p^{(\mathrm{M})}_c([n]^2) \; = \; \frac{\pi^2}{6\log n} \,-\, \frac{1}{(\log n)^{3/2+o(1)}}.$$
\end{conj}

Given a rectangle $R$, we say that a set $A \subset R$ is a \emph{minimal percolating set} if $A$ spans $R$, but no proper subset of $A$ does so (see~\cite{minl}, for example). Given $m \ge n$ and $x \ge 0$, let $F(m,n,x)$ denote the number of minimal percolating sets of size $m + x$ in modified bootstrap percolation on $R = [m] \times [n]$. We remark that Conjecture~\ref{mod} would follow from the method of this paper, together with following bound:
$$F(m,n,x) \; \le \; n^{m - n + 2x + o(n)}.$$
Note that even if we restrict ourselves to `threshold' models, in which a vertex is infected if at least $r$ elements of its neighbourhood are infected, we still run into similar problems. Indeed, consider the model in which a vertex is infected if at least four of its eight neighbours (including diagonals) are infected. A typical seed $R$ is shaped like an octagon, and the number of infected sites used to fill $R$ (in the random process) is roughly $\phi(R)$ (which we define to be the number of external vertices plus the number of external edges), while the minimal number required to span $R$ is only $\phi(R)/2$.

Finally, returning to the standard bootstrap process, recall that Theorem~\ref{sharper} determines the second term of $p_c([n]^2,2)$ up to a $\poly(\log\log n)$-factor. We ask whether this error term can be removed.

\begin{prob}
Determine $\alpha \in [0,3]$, if it exists, such that
$$p_c([n]^2,2) \; = \; \frac{\pi^2}{18\log n} \, - \, \frac{(\log \log n)^{\alpha+o(1)}} {(\log n)^{3/2}}$$
as $n \to \infty$.
\end{prob}

As usual in bootstrap percolation, it would not be unreasonable to suspect that the upper bound in Theorem~\ref{sharper} is closer to the truth.

\end{document}